\documentclass[12pt,a4paper,twoside]{article}

\usepackage[centertags]{amsmath}
\usepackage{amsxtra}
\usepackage{amsfonts}
\usepackage{amssymb}
\usepackage{amsthm}
\usepackage{amscd}
\usepackage{newlfont}
\usepackage{slashed}
\usepackage[dvips]{graphicx}
\usepackage[all]{xy}
\usepackage{pinlabel}

\title{On rational blow-downs in Heegaard-Floer homology}

\author{Maria Michalogiorgaki}

\begin{document}

\maketitle
\newtheorem{theorem}{Theorem}
\newtheorem{proposition}{Proposition}
\newtheorem{corollary}{Corollary}
\newtheorem{lemma}{Lemma}
\newtheorem{remark}{Remark}
\newtheorem{note}{Note}
\newtheorem{definition}{Definition}

\begin{abstract}
Motivated by a result of L.P. Roberts on rational blow-downs in
Heegaard-Floer homology, we study such operations along 3-manifolds
that arise as branched double covers of $S^{3}$ along several
non-alternating, slice knots.
\end{abstract}

\section{Introduction}

In 1993, R. Fintushel and R. Stern introduced the \textit{rational
blow-down} of smooth 4-manifolds, a surgical procedure consisting of
removing the interior of a negative-definite simply connected smooth
4-manifold $C_{p}$ embedded in a closed smooth 4-manifold $X$ and
replacing it with a rational homology ball. They also studied the
effect of this process on both the Donaldson and the Seiberg-Witten
invariants (\cite{FS}). Several years later, J. Park extended their
results to more general configurations $C_{p,q}$, thus defining the
\textit{generalized rational blow-down}, and computed how the
Seiberg-Witten invariants change under this operation (\cite{P2}).
The first technique was used by Park in constructing an exotic
smooth structure on $\mathbb{C}P^{2}\sharp 7
\overline{\mathbb{C}P}^{2}$ (\cite{P}), while the second was applied
by Park, Stipsicz and Szab\'{o} in constructing exotic smooth
structures on $\mathbb{C}P^{2}\sharp 6 \overline{\mathbb{C}P}^{2}$
and $\mathbb{C}P^{2}\sharp 5 \overline{\mathbb{C}P}^{2}$ (\cite{SS}
and \cite{PSS} respectively). More recently, in \cite{SSW} and
\cite{GaS}, D. Gay, A. Stipsicz, Z. Szab\'{o} and J. Wahl extended
further the above procedures, this time along certain
negative-definite plumbing trees, while in \cite{Ro}, L. Roberts
studied the rational blow-down along the branched double cover
$\Sigma(K)$ of $S^{3}$ along alternating, slice knots K in $S^{3}$
and its effect on the Ozsv\'{a}th-Szab\'{o} 4-manifold
invariant.\\

\noindent In this paper, we turn to study the rational blow-down
operation along 3-manifolds that arise as branched double covers of
$S^{3}$ along non-alternating, slice knots. We narrow our attention
down to knots with up to ten crossings and then the knots $8_{20},
9_{46}, 10_{129}, 10_{137}, 10_{140}, 10_{153}$ and $10_{155}$ are
the only ones with the desired properties (see \cite{In} and
\cite{DBN}). First, we study the mirrors of some of these knots,
specifically $\overline{8_{20}}, \overline{9_{46}},
\overline{10_{137}}$ and $\overline{10_{140}}$, which are also
non-alternating and slice.  The branched double covers of $S^{3}$
along these (denoted by $Y_{1},Y_{2},Y_{3}$ and $Y_{4}$
respectively) bound negative-definite 4-manifolds as well as
rational homology balls. The existence of such balls is guaranteed
by the fact that the knots are slice, but we also find explicit
descriptions of them using \cite{CH}. Using the results in
\cite{OS1} and \cite{OS2}, we show that the 3-manifolds $Y_{i}$, $ i
\in \{1,2,3,4\}$, are L-spaces and we then move on to write rational
blow-down formulas along them, applying the results of Roberts in
\cite{Ro}. We note that the 3-manifolds $Y_{i}$, $ i \in
\{1,2,3,4\}$, along which we perform the rational blow-down do not
belong in any of the categories $G_{rat}$, $W$, $N$, $M$, $A$, $B$,
$C$ described in \cite{SSW}. In the next to last section of the
paper, we briefly discuss the cases of the branched double covers of
$S^{3}$ along $10_{129}, 10_{153}$ and $10_{155}$ and in the last
section, we present how the relationship between Heegaard-Floer
homology and Khovanov homology can be used to draw some of the above
conclusions
over $\mathbb{Z}_{2}$ instead of $\mathbb{Q}$.\\

\noindent \textbf{Acknowledgements:} The author wishes to thank Z.
Szab\'{o} for his guidance during the course of this work as well as
J. Rasmussen and J. Greene for several helpful discussions.

\section{Preliminaries}

\subsection{Weighted graphs}

Let us start by introducing some terminology. Consider a graph $G$.
The \textbf{\textit{degree}} of a vertex $v$ of $G$, denoted $d(v)$,
is the number of edges which contain $v$. If $G$ is equipped with an
integer valued function $m$ on its vertices, then it is called a
\textbf{\textit{weighted graph}} and if $v$ is a vertex of such a
graph, then $m(v)$ is called the \textbf{\textit{multiplicity}} of
$v$. A vertex $v$ of a weighted graph $G$ is called
\textbf{\textit{bad}} if
\[
m(v) > -d(v).
\]

A weighted graph $G$ gives rise to a 4-manifold $X(G)$ with boundary
$Y(G)$. $X(G)$ is obtained as follows: On each vertex $v$ of $G$
consider a $D^{2}$-bundle over $S^{2}$ with Euler class $m(v)$.
Whenever two vertices $v$ and $w$ are joined by an edge, "plumb" the
corresponding disk bundles, i.e. pick a small disk ($D_{v}$ and
$D_{w}$) on each sphere so that the disk bundle over it is a product
($D_{v} \times D^{2}$ and $D_{w} \times D^{2}$) and then identify
$D_{v} \times D^{2}$ with $D_{w} \times D^{2}$ using a map that
preserves the product structures but interchanges the factors.

For $X(G)$ as above, $H_{2}(X(G);\mathbb{Z})$ is the lattice spanned
by the vertices of G and if $[v]$ is the homology class
corresponding to the vertex $v$, then the intersection form
$Q_{X(G)}$ is given by: $Q_{X(G)}([v],[v])=m(v)$ and
$Q_{X(G)}([v],[w])=1$ $(0)$ if the vertices $v$ and $w$ are (are
not) connected by an edge.

A weighted graph $G$ is called \textbf{\textit{negative-definite}}
when it is a disjoint union of trees and $Q_{X(G)}$ is
negative-definite.

\subsection{The Ozsv\'{a}th-Szab\'{o} 4-manifold invariant}

We move on to briefly recalling the definition of the closed
4-manifold invariant $\Phi$ introduced in \cite{OS6}.

Consider $W$ a smooth, oriented, connected cobordism between two
connected 3-manifolds $Y_{1}$ and $Y_{2}$ and $\mathfrak{s}$ a
$spin^{c}$ structure on $W$. $W$ and $\mathfrak{s}$ induce maps
$F_{W,\mathfrak{s}}^{+}$, $F_{W,\mathfrak{s}}^{-}$ and
$F_{W,\mathfrak{s}}^{\infty}$ between
$HF^{+}(Y_{1},\mathfrak{t}_{1})$ and
$HF^{+}(Y_{2},\mathfrak{t}_{2})$ (respectively
$HF^{-}(Y_{1},\mathfrak{t}_{1})$ and
$HF^{-}(Y_{2},\mathfrak{t}_{2})$,
$HF^{\infty}(Y_{1},\mathfrak{t}_{1})$ and
$HF^{\infty}(Y_{2},\mathfrak{t}_{2})$), where
$\mathfrak{t}_{i}=\mathfrak{s}|_{Y_{i}}$, $i \in \{1,2\}$. These
maps are uniquely determined up to sign.

If $W$ is a closed 4-manifold, it can be punctured in two points and
the resulting object can be viewed as a cobordism from $S^{3}$ to
$S^{3}$. Under the additional condition that $b_{2}^{+}(W)>1$, this
object can be further cut along some 3-manifold $Y$ and thus divided
into two cobordisms $W_{1}$ and $W_{2}$ with $b_{2}^{+}(W_{i})>0$,
$i \in \{1,2\}$, so that the map
\[
Spin^{c}(W) \to Spin^{c}(W_{1}) \times Spin^{c}(W_{2})
\]
induced by restriction is injective. Then a "mixed invariant"
\[
F_{W,\mathfrak{s}}^{mix}:HF^{-}(Y_{1}, \mathfrak{t_{1}}) \to
HF^{+}(Y_{2}, \mathfrak{t_{2}})
\]
can be defined by combining $F_{W_{1},\mathfrak{s}|_{W_{1}}}^{-}$
and $F_{W_{2},\mathfrak{s}|_{W_{2}}}^{+}$ in an appropriate way
(using the identification $HF^{+}_{red}(Y,\mathfrak{s}|_{Y})\cong
HF^{-}_{red}(Y,\mathfrak{s}|_{Y})$). This "mixed invariant" gives
rise to the \textbf{\textit{invariant $\Phi_{W,\mathfrak{s}}$}},
which is a map
\[
\Phi_{W,\mathfrak{s}}:\mathbb{Z}[U]\otimes \Lambda
^{*}(H_{1}(W)/Tors) \to \mathbb{Z} / \pm 1
\]
and is a smooth, oriented 4-manifold invariant.

\subsection{The correction term $d(Y,\mathfrak{t})$}

The 4-dimensional theory reviewed in the previous section has as a
by-product an absolute rational lift to the relative $\mathbb{Z}$
grading on the Floer-homology groups of a 3-manifold $Y$ endowed
with a torsion $spin^{c}$ structure. The correction term
$d(Y,\mathfrak{t})$ that we discuss in the present section is an
application of these absolute gradings.

In \cite{OS5}, the authors define a $\mathbb{Q}$-valued invariant
$d(Y,\mathfrak{t})$ (also called the \textbf{\textit{correction term
$d(Y,\mathfrak{t})$}}) associated to an oriented rational homology
3-sphere Y equipped with a $spin^{c}$ structure $\mathfrak{t}$ as
follows:

\begin{definition}
$d(Y,\mathfrak{t})$ is the minimal grading ($\widetilde{gr}$) of any
non-torsion element in the image of $HF^{\infty}(Y,\mathfrak{t})$ in
$HF^{+}(Y,\mathfrak{t})$.
\end{definition}

\noindent This is the Heegaard Floer homology analogue of the
Fr{\o}yshov invariant in Seiberg-Witten theory.

In the same paper it is proven that if $Y$ and $\mathfrak{t}$ are as
above and $X$ is a smooth, negative-definite 4-manifold with
$\partial X=Y$, then $\forall$ $\mathfrak{s} \in Spin^{c}(X)$ with
$\mathfrak{s}|_{Y}=\mathfrak{t}$
\[
c_{1}(\mathfrak{s})^{2}+rk(H^{2}(X;\mathbb{Z})) \leq
4d(Y,\mathfrak{t})
\]
In addition, in \cite{OS1}, Corollary 1.5, it is proven that for
negative-definite graphs G with at most two bad vertices
\begin{equation}
\label{d}
d(Y(G)),\mathfrak{t})=max_{\{K \in
Char_{\mathfrak{t}}((G))\}} \frac{K^{2}+|G|}{4}
\end{equation}
\noindent where $Char_{\mathfrak{t}}(G)$ denotes the set of
characteristic vectors for $X(G)$ which are first Chern classes of
$spin^{c}$ structures whose restriction to the boundary is
$\mathfrak{t}$ and $K^{2}$ is computed using $Q_{X(G)}^{-1}$.

\section{Rational blow-down along $Y_{1}=\Sigma(\overline{8_{20}})$}

Denote the branched double cover of $S^{3}$ along
$\overline{8_{20}}$ by $Y_{1}=\Sigma(\overline{8_{20}})$.

\subsection{A negative-definite 4-manifold $W_{1}$ with $\partial(W_{1})=Y_{1}$.}
\label{ss:W}

\indent Consider the knot $\overline{8_{20}}$ and following
\cite{OS2} construct a checkerboard coloring of the plane (see
Figure \ref{1}).

\begin{figure}[h]
\centering
\includegraphics{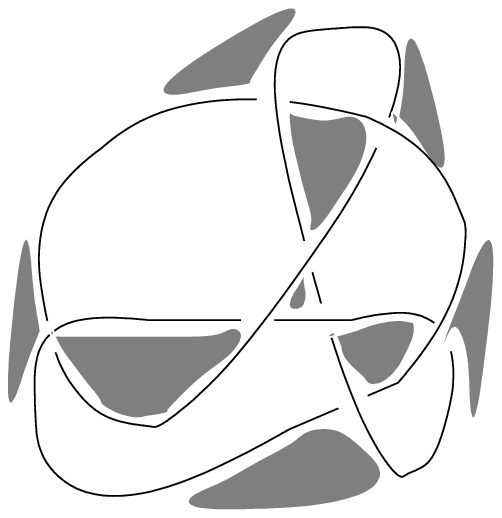}
\caption{A checkerboard coloring for $\overline{8_{20}}$} \label{1}
\end{figure}

Associate 1-handles to the black regions but one (the outer region
in Figure \ref{1}, without loss of generality) and at each crossing
add a $\pm 1$ framed 2-handle to an unknot looping through the two
1-handles using the sign convention of Figure \ref{2}.

\begin{figure}[h]
\centering
\includegraphics{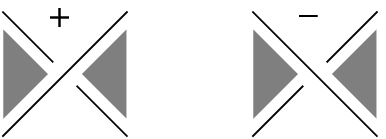}
\caption{Sign convention} \label{2}
\end{figure}

This process gives a 4-manifold $W_{1}$ with boundary
$\Sigma(\overline{8_{20}})$. After appropriate cancelations of
1-handles by 2-handles (see section 5.4 of \cite{GS} for more
details), the 4-manifold
described above can be represented by the plumbing tree below\\

\begin{center}
\setlength{\unitlength}{3cm}
          \begin{picture}(2.1,.6)
          \put(.6,.6){\circle*{.05}}
          \put(.4,.45){$-2$}
          \put(1.1,.6){\circle*{.05}}
          \put(.9,.45){$-2$}
          \put(1.6,.6){\circle*{.05}}
          \put(1.5,.45){$-2$}
          \put(2.1,.6){\circle*{.05}}
          \put(2,.45){$-2$}
          \put(1.1,.1){\circle*{.05}}
          \put(.9,0){$-3$}
          \put(.6,.6){\line(1,0){1.5}}
          \put(1.1,.1){\line(0,1){.5}}
          \put(0,.5){$W_{1}=$}
          \end{picture}
\end{center}

\subsection{$Y_{1}$ is an L-space}

In this subsection, we will exhibit that $Y_{1}$ is an L-space. The
notion of an L-space is a generalization of that of a lens space.
The precise definition is as follows (\cite{OS3}):

\begin{definition}
A closed 3-manifold Y is called an \textbf{\textit{L-space}} if
$H_{1}(Y;\mathbb{Q})=0$ and $\widehat{HF}(Y)$ is a free abelian
group with rank equal to $|H_{1}(Y;\mathbb{Z})|$, the number of
elements in $H_{1}(Y;\mathbb{Z})$.
\end{definition}

\noindent $\widehat{HF}(Y)$ is the 3-manifold invariant defined in
\cite{OS4}.\\

$Y_{1}$ is a rational homology sphere ($\mathbb{Q}HS^{3}$), as is
the branched double cover of $S^{3}$ along any knot K, denoted from
now on as $\Sigma(K)$. This is true because
$|H_{1}(\Sigma(K);\mathbb{Z})|=|\Delta_{K}(-1)|=det(K)$ finite.\\

To prove that $Y_{1}$ is an L-space over $\mathbb{Q}$, we first need
to introduce the notion of a quasi-alternating link, as defined in
\cite{OS2}.

\begin{definition}
The set $Q$ of \textbf{\textit{quasi alternating links}} is the
smallest set of links which satisfies the following properties:
\begin{enumerate}
\item the unknot is in $Q$
\item the set $Q$ is closed under the following operation. Suppose
$L$ is a link which admits a projection with a crossing with the
following properties:
\begin{itemize}
\item both resolutions $L_{0}, L_{1} \in Q$ (see Figure \ref{3})
\item $det(L_{0}),det(L_{1})\neq 0$
\item $det(L)=det(L_{0})+det(L_{1})$
\end{itemize}
then $L \in Q$.
\end{enumerate}
\end{definition}


\begin{figure}[h]
\labellist \small\hair 2pt \pinlabel $L$ at 22 0 \pinlabel $L_{0}$
at 96 0 \pinlabel $L_{1}$ at 168 0
\endlabellist
\centering
\includegraphics{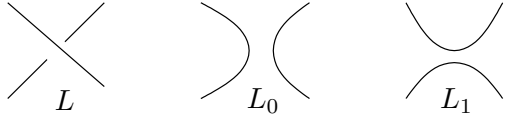}
\caption{The 0-and 1-resolutions of a link L at one of its crossings
as above are obtained by replacing this crossing with the simplified
pictures shown in the figure.} \label{3}
\end{figure}

\begin{proposition}
\label{quasi820}
$\overline{8_{20}}$ is a quasi-alternating knot.
\end{proposition}

\begin{proof}

Consider a projection of $\overline{8_{20}}$ like the one in Figure
\ref{1} and resolve the left topmost crossing. It is easy to see
that the 1-resolution yields the unknot and the 0-resolution yields
the link shown in Figure \ref{4}, call it L.

\begin{figure}[h]
\centering
\includegraphics{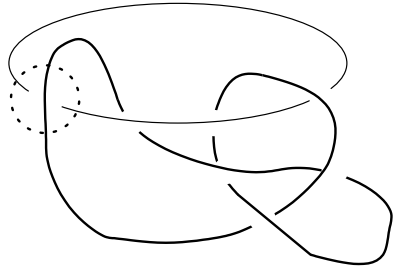}
\caption{The 0-resolution of $8_{20}$ at it's left topmost crossing.}
\label{4}
\end{figure}

Using the skein relationship
$\Delta_{L_{+}}(x)-\Delta_{L_{-}}(x)+(x^{-\frac{1}{2}}-x^{\frac{1}{2}})\Delta_{L_{0}}(x)=0$
(Figure \ref{5} conveys the meaning of $L_{+},L_{-}$ and $L_{0}$)
satisfied by the Conway normalized Alexander polynomial (see
\cite{Li}), one can compute that
$\Delta_{L}(x)=x^{-\frac{3}{2}}-3x^{-\frac{1}{2}}+3x^{\frac{1}{2}}-x^{\frac{3}{2}}$
and $det(L)=|\Delta_{L}(-1)|=8$.

\begin{figure}[h]
\labellist \small\hair 2pt \pinlabel $L_{+}$ at 22 -4 \pinlabel
$L_{-}$ at 102 -4 \pinlabel $L_{0}$ at 185 -4
\endlabellist
\centering
\includegraphics{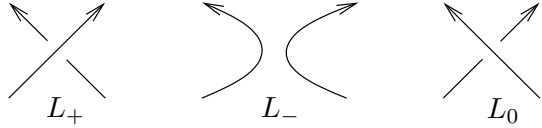}
\caption{$L_{+},L_{-}$ and $L_{0}$ in the skein relationship for the
Alexander polynomial}
\label{5}
\end{figure}


Thus, it remains to prove that L is in it's turn a quasi-alternating
link. To this end, we resolve the marked crossing in Figure \ref{4}
and we get the unknot as the 1-resolution and the knot
$\overline{5_{2}}$ as the 0-resolution. The result folows, since
$\overline{5_{2}}$ is alternating, thus quasi-alternating (See Lemma
3.2 of \cite{OS2}), and
$det(\overline{5_{2}})=|\Delta_{\overline{5_{2}}}(-1)|=|\Delta_{5_{2}}(-1)|=|2(-1)^{-1}-3+2(-1)|=7$.

\end{proof}

\begin{corollary}
$ Y_{1}=\Sigma(\overline{8_{20}}) $ is an L-space
\end{corollary}

\begin{proof}
Making use of a result in \cite{OS2} that states that if $L$ is a
quasi-alternating link, then $\Sigma(L)$ is an L-space, Proposition
\ref{quasi820} implies that $ Y_{1}=\Sigma(\overline{8_{20}}) $ is
an L-space.
\end{proof}

\subsection{$Y_{1}$ bounds a rational homology ball $B_{1}$}
\label{ss:B}

To see this, it suffices to notice that $\overline{8_{20}}$ is slice
and the branched cover of $B^{4}$ along the slice disk is a rational
homology ball. We mention here that whether a knot is slice or not
can be read off from its smooth four genus, which has been computed
for all knots up to ten crossings and is listed on the corresponding
knot tables. For the specific case of the knot $8_{20}$ that we are
studying here, the interested reader is refered to page 86 of
\cite{Li} for a concrete description of the slice disc.

In fact, $Y_{1}$ is one of the manifolds listed in \cite{CH}, it is
the manifold $(2,3,3;9)$ in category (5) with p=3 and s=-1. Thus, we
can explicitly describe a 2-handle addition to $Y_{1} \times I$
along a circle in $Y_{1} \times 1$ that leads to a manifold with
boundary $Y_{1} \bigcup S^{1} \times S^{2}$ and eventually, after
attaching a 3- and a 4-handle, to a rational homology ball $B_{1}$
with boundary $Y_{1}$. We proceed to do so.

First note that $Y_{1}$ can be alternatively represented as in
Figure \ref{6}.

\begin{figure}[h]
\centering
\includegraphics{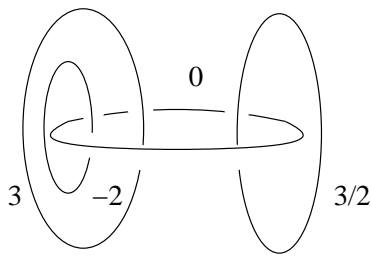}
\caption{An alternative description of $Y_{1}$}
\label{6}
\end{figure}

Then, there is a 2-handle addition depicted in the first part of
Figure \ref{14} at the end of the paper that has as end product the
last manifold of this figure, which, according to the Lemma in page
26 of \cite{CH}, is homeomorphic to $S^{1} \times S^{2}$.

\subsection{Blow-down formula}
\label{ss:bd}

We will call our graph $G_{1}$ and label its vertices as shown below

\begin{center}
 \setlength{\unitlength}{3cm}
          \begin{picture}(2.1,.6)
          \put(.6,.6){\circle*{.05}}
          \put(.5,.45){$v_{1}$}
          \put(1.1,.6){\circle*{.05}}
          \put(.9,.45){$v_{2}$}
          \put(1.6,.6){\circle*{.05}}
          \put(1.5,.45){$v_{3}$}
          \put(2.1,.6){\circle*{.05}}
          \put(2,.45){$v_{4}$}
          \put(1.1,.1){\circle*{.05}}
          \put(.9,0){$v_{5}$}
          \put(.6,.6){\line(1,0){1.5}}
          \put(1.1,.1){\line(0,1){.5}}
          \put(0,.5){$G_{1}=$}
          \end{picture}
\end{center}

\noindent Note that $G_{1}$ has only one bad vertex and that is
$v_{2}$ with $-2=m(v_{2})>-d(v_{2})=-3$.

First, we will make use of the calculations of the Heegaard Floer
homology groups for 3-manifolds obtained by plumbings of spheres
specified by certain graphs carried out in section 3 of \cite{OS1}.
It is easy to check that of the 48 characteristic vectors $K_{0} \in
\{0,2\} \times \{0,2\} \times \{0,2\} \times \{0,2\} \times
\{-1,1,3\}$, after applying the algorithm described in section 3 of
\cite{OS1}, only 9 initiate a path ending at a vector $L$ satisfying
\begin{equation}\label{16.1}
-2 \leq \langle L,v_{i} \rangle \leq 0 \mbox{ } \forall \mbox{ } i
\in \{1,2,3,4\}
\end{equation}
\begin{equation}\label{16.2}
-3 \leq  \langle L,v_{5} \rangle \leq 1
\end{equation}
These are the following: $(0,0,0,0,\pm 1), (2,0,0,0,\pm 1),
(0,0,0,2,\pm 1),$ \\
$(0,0,0,0,3), (0,0,2,0,-1)$ and $(0,2,0,0,-1)$.

\begin{remark}
This, together with the fact that $|H_{1}(Y_{1})| = 9$,
provides an \textbf{\textit{alternative proof of the fact that
$Y_{1}$ is an L-space}}.
\end{remark}

Next, we need to check which of the above 9 vectors representing
$spin^{c}$ structures on $Y_{1}$ extend to the rational homology
ball. Figure \ref{7} illustrates the first few steps of computing
the enhanced intersection form after the 2-handle addition that we
presented at the end of section 3.3. We leave it as an exercise to
the reader to carry out the next few steps and we only record here
the outcome of this process:

\[
A_{1}=
\begin{pmatrix}
-2 & 1 & 0 & 0 & 0 & 0 \\
1 & -2 & 1 & 0 & 1 & 0 \\
0 & 1 & -2 & 1 & 0 & 0 \\
0 & 0 & 1 & -2 & 0 & -1 \\
0 & 1 & 0 & 0 & -3 & -2 \\
0 & 0 & 0 & -1 & -2 & -4
\end{pmatrix}
\]

\noindent with $Ker(A_{1})=<(-1,-2,-\frac{5}{3}, -\frac{4}{3},
-\frac{4}{3},1)>.$ The $spin^{c}$ structures that extend are
represented by vectors that are orthogonal to the kernel of the
enhanced intesection form, that is satisfy
\begin{equation}
(a_{1},a_{2},a_{3},a_{4},a_{5},a_{6})(-1,-2,-\frac{5}{3},
-\frac{4}{3}, -\frac{4}{3},1)=0
\end{equation}
It is easy to see that only 3 of these vectors, specifically
(0,0,0,0,3),(0,0,0,2,1) and (0,0,2,0,-1), satisfy the equation
\begin{equation}
2a_{3}+a_{4}+a_{5} \equiv 0(mod3)
\end{equation}
and thus can be extended to the rational homology ball.

\begin{figure}[h]
\centering
\includegraphics{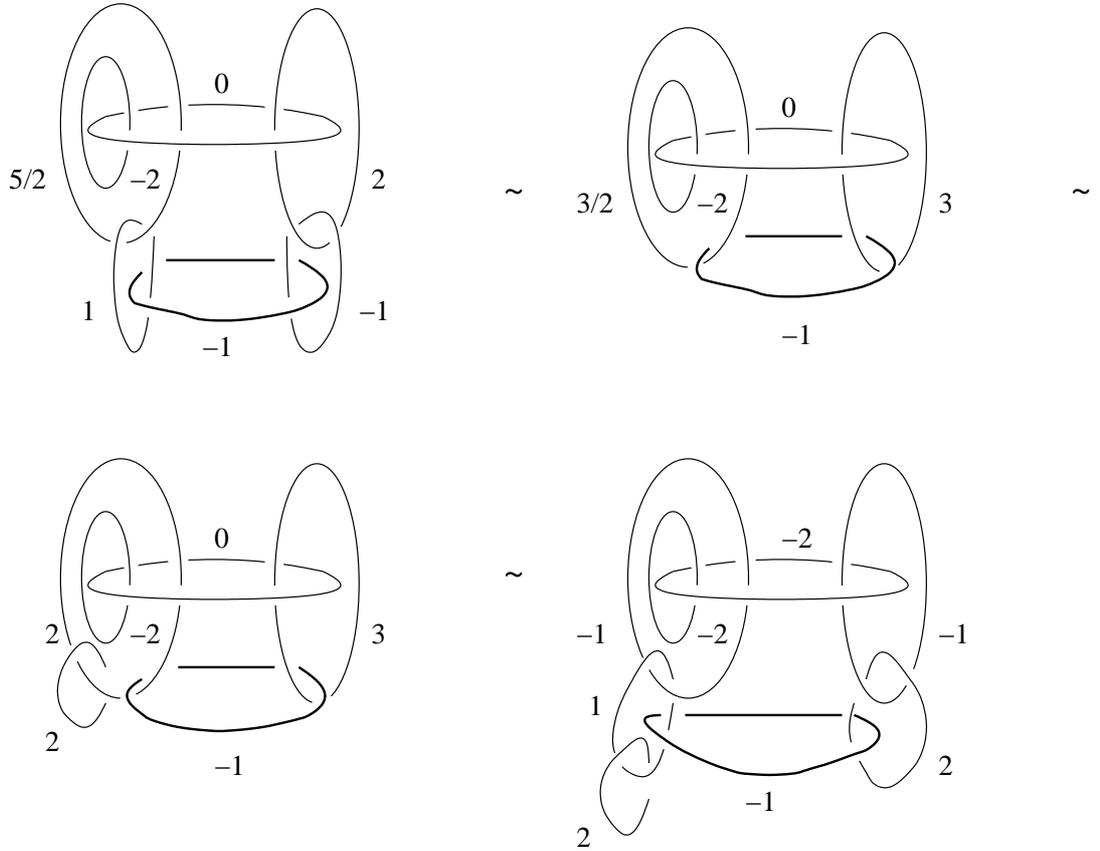}
\caption{Computing the enhanced intersection form} \label{7}
\end{figure}

\begin{proposition}
\label{XB1}
Let $X_{W_{1}}$ be a closed, oriented, smooth 4-manifold
with $b_{2}^{+}(X_{W_{1}}) > 1$ containing $W_{1}$ and let
$\mathfrak{s}_{i},$ $i \in \{1,2,3\}$, $spin^{c}$ structures on
$X_{W_{1}}$ that restrict to $W_{1}$ to give the three $spin^{c}$
structures listed above. Then, for any $\phi \in Diff^{+}(Y_{1})$
the 4-manifold $X_{B_{1}}=(X_{W_{1}}-W_{1})\bigcup _{\phi} B_{1}$
has $spin^{c}$ structures $\mathfrak{s'}_{i}$, $i \in \{1,2,3\}$,
for
which\\

\begin{center}
$\Phi _{(X_{B_{1}}, \mathfrak{s'}_{i})}= \pm
\Phi_{(X_{W_{1}},\mathfrak{s}_{i})}$
\end{center}

\end{proposition}

\begin{proof}
Call $\mathfrak{t}_{i}=\mathfrak{s}_{i}|_{Y_{1}}$. We will apply
Theorem 2 of \cite{Ro}, so we only need to check that $Y_{1}$
satisfies the three conditions listed there. We have already seen
that $Y_{1}$ is a rational homology sphere, so it remains to check
the last two conditions. For the second condition, recall that there
is a long exact sequence
\[
... \rightarrow HF^{-}(Y_{1},\mathfrak{t}_{i})
\xrightarrow{\text{i}} HF^{\infty}(Y_{1},\mathfrak{t}_{i})
\xrightarrow{\text{$\pi$}} HF^{+}(Y_{1},\mathfrak{t}_{i})
\xrightarrow{\text{$\delta$}} ...
\]
and that the 3-manifold invariant $HF^{+}_{red}(Y_{1},
\mathfrak{t}_{i})$ is defined as $HF^{+}_{red}(Y_{1},
\mathfrak{t}_{i})=Coker(\pi)=HF^{+}(Y_{1},\mathfrak{t}_{i})/Im\pi$.
But $Y_{1}$ is an L-space, so the map $\delta$ is trivial,
$Im\pi=Ker\delta=HF^{+}(Y_{1}, \mathfrak{t}_{i})$ and
$HF^{+}_{red}(Y_{1}, \mathfrak{t}_{i})=0$. Hence the second
condition of the theorem holds as well. Finally, the third condition
requires $W_{i}$ to be a sleek negative-definite 4-manifold. This is
true according to the results in \cite{OS1}, given that $G_{1}$ is a
negative-definite graph with only one bad vertex.
\end{proof}

\subsection{Using $d(Y,\mathfrak{t})$ to study the rational blow-down operation}

At this point, we present a different approach to studying which of
the $spin^{c}$ structures extend to the rational homology ball. The
advantage of this approach is that it does not require a concrete
description of the rational homology ball.

Fix a $spin^{c}$ structure $\mathfrak{t}$ over $Y_{1}$. Then
\begin{equation}
d(Y_{1},\mathfrak{t})=max_{\{K \in Char_{\mathfrak{t}}(G_{1})\}}
\frac{K^{2}+5}{4}
\end{equation}
since the graph $G_{1}$ that we are presently studying has only one
bad vertex. By Proposition 3.2 of \cite{OS1} the maximum is always
achieved among the characteristic vectors in
$Char_{\mathfrak{t}}(G_{1})$ which have coordinates in $\{0,2\}
\times \{0,2\} \times \{0,2\} \times \{0,2\} \times \{-1,1,3\}$ and
initiate paths with final vectors satisfying the equations
(\ref{16.1}) and (\ref{16.2}).

Furthermore, if a $spin^{c}$ structure $\mathfrak{t}$ extends across
a rational homology ball, then $d(Y_{1},\mathfrak{t})=0$. This
follows from the more general statement proven in Proposition 9.9 of
\cite{OS5} that if $(Y_{1},\mathfrak{t}_{1})$ and
$(Y_{2},\mathfrak{t}_{2})$ are rational homology cobordant rational
homology 3-spheres equipped with $spin^{c}$ structures, then
$d(Y_{1},\mathfrak{t}_{1})=d(Y_{2},\mathfrak{t}_{2})$.

We compute the square of the nine vectors above and only
$(0,0,0,0,3)$, $(0,0,0,2,1)$ and $(0,0,2,0,-1)$ have square equal to
-5. Therefore these give the only candidates for $spin^{c}$
structures that extend.

Moreover, we were already expecting precisely
$3=\sqrt{9}=\sqrt{|H_{1}(Y_{1};\mathbb{Z})|}$ $spin^{c}$ structures
to extend, according to the arguments presented in Lemma 2 of
\cite{Ro}. We restate this here and then use it to justify our
claim.

\begin{lemma}
$Y$ is a rational homology 3-sphere and $h=|H_{1}(Y;\mathbb{Z})|$.
$Y$ bounds $X$ and $s=|det(Q_{X})|$, $Q_{X}$ denoting the
intersection form of $X$. Then $h=st^{2}$ where $t$ is the order of
the image of the torsion of $H^{2}(X;\mathbb{Z})$ in
$H^{2}(Y;\mathbb{Z})$ (and $st$ is the order of the image of
$H^{2}(X;\mathbb{Z})$ in $H^{2}(Y;\mathbb{Z})$).
\end{lemma}

\noindent Applying this for $X=B_{1}$ and $Y=Y_{1}$ and setting
$s=1$ since $b_{2}(B_{1})=0$ gives that
$9=|H_{1}(Y_{1};\mathbb{Z})|=t^{2}$ and so $t=3$, i.e. the order of
the image of the torsion of $H^{2}(B_{1};\mathbb{Z})$ in
$H^{2}(Y_{1};\mathbb{Z})$ is 3.

The arguments in the preceding two paragraphs verify the answer we
got using the enhanced intersection form.

\section{Rational blow-down along $Y_{2}=\Sigma(\overline{9_{46}})$}

Denote the branched double cover of $S^{3}$ along
$\overline{9_{46}}$ by $Y_{2}=\Sigma(\overline{9_{46}})$. Following
a process analogous to that of subsection 1.1, we construct a
negative-definite 4-manifold $W_{2}$ with $\partial(W_{2})=Y_{2}$.
This is depicted
below: \\

\begin{center}
\setlength{\unitlength}{3cm}
          \begin{picture}(2.6,.6)
          \put(.6,.6){\circle*{.05}}
          \put(.4,.45){$-2$}
          \put(1.1,.6){\circle*{.05}}
          \put(.9,.45){$-2$}
          \put(1.6,.6){\circle*{.05}}
          \put(1.4,.45){$-2$}
          \put(2.1,.6){\circle*{.05}}
          \put(2,.45){$-2$}
          \put(2.6,.6){\circle*{.05}}
          \put(2.5,.45){$-2$}
          \put(1.6,.1){\circle*{.05}}
          \put(1.4,0){$-3$}
          \put(.6,.6){\line(1,0){2}}
          \put(1.6,.1){\line(0,1){.5}}
          \put(0,.5){$W_{2}=$}
          \end{picture}
\end{center}

Using the algorithm presented in subsection 1.4, we compute that
only 9 of the 96 characteristic vectors $K_{0} \in \{0,2\} \times
\{0,2\} \times \{0,2\} \times \{0,2\} \times \{0,2\} \times
\{-1,1,3\}$ initiate paths that terminate in a vector $L$ satisfying

\begin{equation}\label{16.1'}
-2 \leq \langle L,v_{i} \rangle \leq 0 \mbox{ } \forall \mbox{ } i
\in \{1,2,3,4,5\}
\end{equation}
\begin{equation}\label{16.2'}
-3 \leq  \langle L,v_{6} \rangle \leq 1
\end{equation}

\noindent where $v_{1},v_{2},v_{3},v_{4},v_{5}$ are the vertices
with multiplicity -2 of the graph above enumerated from left to
right and $v_{6}$ is the bottom vertex of the same graph. These
vectors are $(0,0,0,0,0,\pm 1),$ $(0,0,0,0,0,3),$ $(0,0,0,0,2,-1)$,
$(2,0,0,0,0,-1),$ $(0,0,0,2,0,-1),$ $(0,2,0,0,0,-1),$
$(2,0,0,0,0,1),$ $(0,0,0,0,2,1)$. Since
$|H_{1}(Y_{2};\mathbb{Z})|=det(9_{46})=9$, we deduce that $Y_{2}$ is
an L-space.

\begin{remark}
In a recent paper (\cite{MO}), C. Manolescu and P. Ozsv\'{a}th  show
that all but 2 ($8_{19}$ and $9_{42}$)of the 85 prime knots with up
to nine crossings are quasi-alternating, which implies that the
corresponding branched double covers of $S^{3}$ are L-spaces.
\end{remark}

Lastly, we know that $Y_{2}$ bounds a rational homology ball since
$\overline{9_{46}}$ is slice and thus we can move on to write a
blow-down formula along $Y_{2}$.

To this end, we compute the squares of the above 9 vectors. It turns
out that 5 of them have square -6, thus giving $d=0$ for the
corresponding $spin^{c}$ structures. These are the vectors
$3v_{6}=(0,0,0,0,0,3)$, $2v_{4}-v_{6}=(0,0,0,2,0,-1)$,
$2v_{5}-v_{6}=(0,0,0,0,2,1)$, $2v_{2}-v_{6}=(0,2,0,0,0,-1)$ and
$2v_{1}-v_{6}=(2,0,0,0,0,1)$. This comes as no surprise, if we take
into account the symmetry of $Y_{2}$ obvious from the plumbing
diagram of $W_{2}$ above (consider the triads
$\{3v_{6},2v_{4}-v_{6},2v_{5}-v_{6}\}$ and
$\{3v_{6},2v_{2}-v_{6},2v_{1}-v_{6}\}$) and for these $spin^{c}$
structures, one can write down a blow-down formula analogous to that
of Proposition \ref{XB1}.

\begin{remark}
\label{r:B2}
An alternative description of $Y_{2}$ is given in
Figure \ref{8}.

\begin{figure}[h]
\centering
\includegraphics{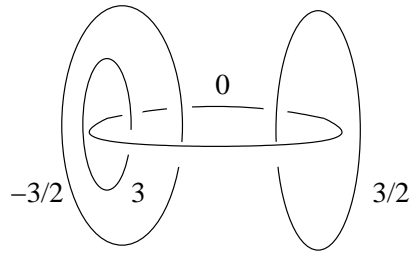}
\caption{An alternative description of $Y_{2}$} \label{8}
\end{figure}

\noindent Using this, we find that $Y_{2}$ is the manifold
$(3,3,3;9)$ of category (3) of the main theorem in \cite{CH} with
$p=q=3$ and $s=0$ and we can compute that the 2-handle addition
shown in Figure \ref{9} yields one way to construct a rational
homology ball bounded by $Y_{2}$.
\begin{figure}[h]
\centering
\includegraphics{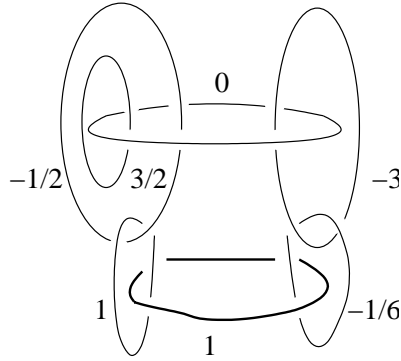}
\caption{2-handle addition for $Y_{2}$} \label{9}
\end{figure}
However, we will not need to make use of this in order to write down
the blow-down formula in this case.
\end{remark}

\section{Rational blow-down along $Y_{3}=\Sigma(\overline{10_{137}})$}

$W_{3}$ below is a negative-definite 4-manifold with $\partial
W_{3}=Y_{3}$.

\begin{center}
\setlength{\unitlength}{3cm}
          \begin{picture}(2.1,1.1)
          \put(.6,1.1){\circle*{.05}}
          \put(.4,.95){$-2$}
          \put(1.1,1.1){\circle*{.05}}
          \put(.9,.95){$-2$}
          \put(1.6,1.1){\circle*{.05}}
          \put(1.5,.95){$-3$}
          \put(2.1,1.1){\circle*{.05}}
          \put(2,.95){$-2$}
          \put(1.1,.6){\circle*{.05}}
          \put(.9,.5){$-2$}
          \put(1.1,.1){\circle*{.05}}
          \put(.9,0){$-3$}
          \put(.6,1.1){\line(1,0){1.5}}
          \put(1.1,.1){\line(0,1){1}}
          \put(0,1){$W_{3}=$}
          \end{picture}
\end{center}

$Y_{3}$ is an L-space, since
$|H_{1}(Y_{3};\mathbb{Z})|=det(10_{137})=25$ and out of the 144
characteristic vectors $K_{0} \in \{0,2\} \times \{0,2\} \times
\{-1,1,3\} \times \{0,2\} \times \{0,2\} \times \{-1,1,3\}$ exactly
25 initiate paths ending in a vector L satisfying

\begin{equation}\label{16.1''}
-2 \leq \langle L,v_{i} \rangle \leq 0 \mbox{ } \forall \mbox{ } i
\in \{1,2,4,5\}
\end{equation}
\begin{equation}\label{16.2''}
-3 \leq  \langle L,v_{i} \rangle \leq 1 \mbox{ } \forall \mbox{ } i
\in \{3,6\}
\end{equation}

\noindent with $v_{1}$, $v_{2}$, $v_{3}$, $v_{4}$ the vertices on
the horizontal part of the graph from left to right and $v_{5}$,
$v_{6}$ the remaining two vertices on the vertical part from top to
bottom. We list these vectors here: \\
 $(0,0,\pm 1,0,0,\pm 1)$, $(0,0,\pm 1,0,0,3)$, $(0,0,\pm 1,2,0,\pm 1)$, $(0,0,-1,2,0,3)$, \\
 $(0,0,-1,0,2,\pm 1)$, $(0,0,1,2,0,-1)$, $(0,2,-1,0,0,\pm 1)$, $(2,0,\pm 1, 0,0 \pm1)$, \\
 $(2,0,-1,2,0,\pm 1)$, $(0,0,3,0,0,\pm 1)$, $(0,0,-1,2,2,-1)$.

Finally, $\overline{10_{137}}$ is slice and therefore $Y_{3}$ bounds
a rational homology ball.

We compute the squares of the above listed vectors and it turns out
that precisely 5 ($=\sqrt{25}$) of them have square equal to -6.
These are $(0,0,1,0,0,3),$ $(0,0,-1,0,2,1),$ $(0,0,1,2,0,1),$
$(0,0,3,0,0,-1),$ $(0,0,-1,2,2,-1)$ and according to arguments
presented in section 3.5, these give the $spin^{c}$ structures that
extend to the rational homology ball.

For them, we can write a blow-down formula similar to that of
Proposition \ref{XB1}.

\begin{remark}
\label{r:B3}
$Y_{3}$ is also among the manifolds listed in
\cite{CH}, in particular it is the manifold (2,5,5;25) in category
(5) with p=5 and s=-1.

\begin{figure}[h]
\centering
\includegraphics{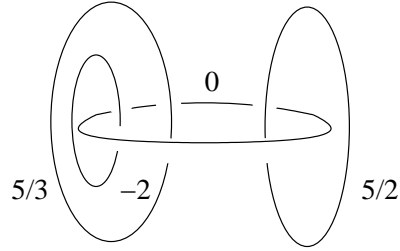}
\caption{The manifold $Y_{3}$} \label{10}
\end{figure}

Figures \ref{10} and \ref{11} present the manifold $Y_{3}$ and a
surgery that leads to the construction of a rational homology ball
$B_{3}$ with $\partial(B_{3})=Y_{3}$ respectively.

\begin{figure}[h]
\centering
\includegraphics{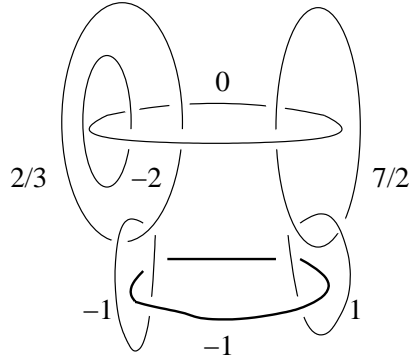}
\caption{2-handle addition for $Y_{3}$} \label{11}
\end{figure}

The enhanced intersection form in this case is given by:

\[
A_{3}=
\begin{pmatrix}
-2 & 1 & 0 & 0 & 0 & 0 & 0\\
1 & -2 & 1 & 0 & 1 & 0 & 0 \\
0 & 1 & -3 & 1 & 0 & 0 & -1 \\
0 & 0 & 1 & -2 & 0 & 0 & -1 \\
0 & 1 & 0 & 0 & -2 & 1 & 0 \\
0 & 0 & 0 & 0 & 1 & -3 & -2 \\
0 & 0 & -1 & -1 & 0 & -2 & -5
\end{pmatrix}
\]

\noindent with $Ker(A_{3})=<(-1,-2,-\frac{7}{5}, -\frac{6}{5},
-\frac{8}{5},-\frac{6}{5},1)>.$

From our list of 25 vectors satisfying (\ref{16.1''}) and
(\ref{16.2''}) only 5 are orthogonal to $Ker(A_{3})$, the same 5
that have square equal to -6. This verifies the conclusions of the
first part of our exposition on blowing-down along $Y_{3}$.

\end{remark}

\section{Rational blow-down along $Y_{4}=\Sigma(\overline{10_{140}})$}

A negative-definite 4-manifold $W_{4}$ with $\partial W_{4}=Y_{4}$
is depicted in the next picture.

\begin{center}
 \setlength{\unitlength}{3cm}
          \begin{picture}(3.1,.6)
          \put(.6,.6){\circle*{.05}}
          \put(.4,.45){$-2$}
          \put(1.1,.6){\circle*{.05}}
          \put(.9,.45){$-2$}
          \put(1.6,.6){\circle*{.05}}
          \put(1.4,.45){$-2$}
          \put(2.1,.6){\circle*{.05}}
          \put(2,.45){$-2$}
          \put(2.6,.6){\circle*{.05}}
          \put(2.5,.45){$-2$}
          \put(3.1,.6){\circle*{.05}}
          \put(3,.45){$-2$}
          \put(1.6,.1){\circle*{.05}}
          \put(1.4,0){$-3$}
          \put(.6,.6){\line(1,0){2.5}}
          \put(1.6,.1){\line(0,1){.5}}
          \put(0,.5){$W_{4}=$}
          \end{picture}
 \end{center}

Label the vertices of the graph above as $v_{1},...,v_{7}$ starting
with those with multiplicity -2 from left to right and finishing at
the vertex with multiplicity -3. Among the 192 characteristic
vectors $K_{0} \in \{0,2\} \times \{0,2\} \times \{0,2\} \times
\{0,2\} \times \{0,2\} \times \{0,2\} \times \{-1,1,3\}$ 9 initiate
paths ending in a vector L satisfying
\begin{equation}\label{16.1'''}
-2 \leq \langle L,v_{i} \rangle \leq 0 \mbox{ } \forall \mbox{ } i
\in \{1,...,6\}
\end{equation}
\begin{equation}\label{16.2'''}
-3 \leq  \langle L,v_{7} \rangle \leq 1
\end{equation}
They are the vectors $(0,0,0,0,0,0,\pm1 )$, $(0,0,0,0,0,0,3)$,
$(2,0,0,0,0,0,\pm 1)$, $(0,2,0,0,0,0,-1)$, $(0,0,0,0,2,0,-1)$,
$(0,0,0,0,0,2,\pm 1)$. Since
$|H_{1}(Y_{4};\mathbb{Z})|=det(10_{140})=9$, we conclude that
$Y_{4}$ is an L-space.

Once again, this particularly nice structure of our 3-manifold $Y$
allows us to write a blow-down formula along it.
($\overline{10_{140}}$ being slice guarantees the existence of a
$\mathbb{Q}HB^{4}$ with boundary $Y_{4}$.) To study which of the
$spin^{c}$ structures on $Y_{4}$ extend to the $\mathbb{Q}HB^{4}$,
we can compute $d$ for the 9 vectors on our list. $(0,0,0,0,0,0,3)$,
$(2,0,0,0,0,0,1)$, $(0,2,0,0,0,0,-1)$ are the only ones with square
equal to -7 and for the three $spin^{c}$ structures corresponding to
them we can write the blow-down formula.

\begin{remark}
\label{r:B4}
$Y_{4}$ is the manifold (3,3,4;9) in the notation of
Casson and Harer in \cite{CH}. It belongs to category (3) with p=3,
q=4 and s=0. Figures \ref{12} and \ref{13} suggest how to construct
a rational homology ball $B_{4}$ with $\partial B_{4}=Y_{4}$.

\begin{figure}[h]
\centering
\includegraphics{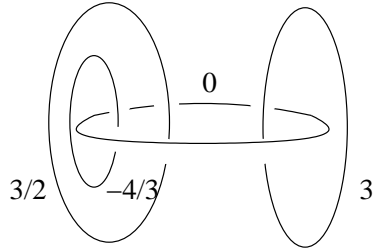}
\caption{The manifold $Y_{4}$} \label{12}
\end{figure}

\begin{figure}[h]
\centering
\includegraphics{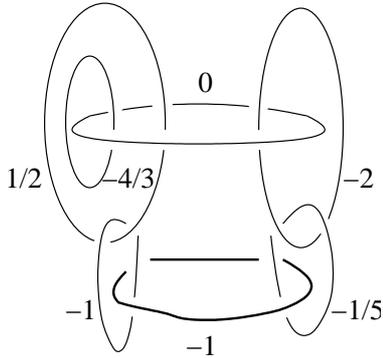}
\caption{2-handle addition for $Y_{4}$} \label{13}
\end{figure}
\end{remark}

\section{The knots $10_{129}$, $10_{153}$ and $10_{155}$.}

The cases of the remaining three knots among the seven listed in the
introduction (i.e. $10_{129}$, $10_{153}$ and $10_{155}$) are still
to be studied, since the techniques used in this paper are
inconclusive for these examples.

\section{Addendum}

In this last section, we discuss some conclusions regarding our
constructions that can be drawn from the relationship between
Heegaard-Floer homology and Khovanov homology.

\subsection{Background in Khovanov homology}
In \cite{Kho}, M. Khovanov presented an algorithm that computes an
invariant of knots and links. Given a link $L$, this invariant is a
bigraded homology theory $Kh(L)$ (strictly speaking cohomology
theory, since the boundary map increases the homological grading by
1) that categorifies the Jones polynomial, in the sense that its
graded Euler characteristic is the unnormalized Jones polynomial of
the link:
\[
\sum_{i,j \in
\mathbb{Z}}(-1)^{i}q^{j}dim(Kh^{i,j}(L))=\widehat{J}(L).
\]
We confine our presentation here to mentioning that the starting
point in defining $Kh(L)$ is to use the state-sum expression for the
unnormalized Jones polynomial $\widehat{J}$.

In addition to the homology groups $Kh^{i,j}(L)$, \textit{reduced
homology} groups $\widetilde{Kh}^{i,j}(L)$ can be defined by
tensoring the original chain complex with $Q$, where $Q=A/XA$ is the
one-dimensional representation of the base ring $A$, to obtain a
reduced chain complex. The Euler characteristic of $\widetilde{Kh}$
is the Jones polynomial:
\[
\sum_{i,j \in
\mathbb{Z}}(-1)^{i}q^{j}dim(\widetilde{Kh}^{i,j}(L))=J(L).
\]

For our purposes, we are interested in the category of H-thin knots:
\begin{definition}
A knot K is called \textbf{\textit{homologically thin}} or
\textbf{\textit{H-thin}} if its nontrivial groups $Kh^{i,j}(K)$ lie
on two adjacent diagonals. By a diagonal we mean a line $2i-j=k$,
for some $k$.
\end{definition}
\noindent As it turns out (\cite{BN}), all but 12 of the 249 knots
with at most 10 crossings are H-thin and for these knots the
homology groups are supported on the diagonals $j-2i=\sigma \pm 1$,
where $\sigma$ denotes the signature of the knot. Moreover, both the
Jones and the Alexander polynomials are alternating and the groups
$\widetilde{Kh}^{i,j}(L)$ lie on one diagonal. Consequently, for
these knots the dimensions of $\widetilde{Kh}^{i,j}(L)$ are given by
the absolute values of the coefficients of $J(L)$ (\cite{Kho2}).

\subsection{The conclusions}
We make the following observations concerning the mirror image
$\overline{K}$ of a slice, H-thin knot $K$. According to \cite{Kho},
for $K$ oriented knot and integers $i,j$, there are equalities of
isomorphism classes of abelian groups
\begin{equation}\label{K mirror}
Kh^{i,j}(\overline{K}) \otimes \mathbb{Q} = Kh^{-i,-j}(K) \otimes
\mathbb{Q}.
\end{equation}
\begin{equation}\label{K mirror 2}
Tor(Kh^{i,j}(\overline{K}))=Tor(Kh^{1-i,-j}(K))
\end{equation}
We saw in the previous section that for an H-thin knot $K$ the
Khovanov homology is supported in the diagonals $j-2i=\sigma+1$ and
$j-2i=\sigma-1$. If $K$ is also slice, then it has signature
$\sigma=0$ and so the Khovanov homology of an H-thin and slice knot
$K$ is supported in the diagonals $j-2i=1$ and $j-2i=-1$. From
equation (\ref{K mirror}) it follows that the non-torsion part of
$Kh(\overline{K})$ is also supported in the same diagonals. From
equation (\ref{K mirror 2}) it follows that the torsion part of
$Kh(K)$ in the $j-2i=-1$ diagonal gives a torsion part in the same
diagonal for $Kh(\overline{K})$ while the torsion part of $Kh(K)$ in
the $j-2i=1$ diagonal gives a torsion part in the line $j-2i=-3$ for
$Kh(\overline{K})$.

For our examples of knots, that is $8_{20},$  $9_{46},$ $10_{137}$
and $10_{140}$, we read off from the knot table (\cite{DBN}) that
these knots are H-thin and that torsion appears only on the
$j-2i=-1$ diagonal and thus we can conclude that the mirror knots
$\overline{8_{20}},$ $\overline{9_{46}},$ $\overline{10_{137}}$ and
$\overline{10_{140}}$ are also H-thin.

\begin{proposition}
\label{Z2Lspace}
 $Y_{i}$ is an L-space over $\mathbb{Z}_{2}$, $i \in
\{1,2,3,4\}$.
\end{proposition}

\begin{proof}
We start by studying $Y_{1}=\Sigma(\overline{8_{20}})$. We know that
\[
9=det(\overline{8_{20}})= |H_{1}(\Sigma(\overline{8_{20}}))| =
|H^{2}(\Sigma(\overline{8_{20}}))|.
\]
Also,
\[
|H^{2}(\Sigma(\overline{8_{20}}))| =
|Spin^{c}(\Sigma(\overline{8_{20}}))|
\]
because there is an isomorphism between
$H^{2}(\Sigma(\overline{8_{20}}))$ and
$Spin^{c}(\Sigma(\overline{8_{20}}))$ (see \cite{GS} for a
discussion on this) and
\[
|Spin^{c}(\Sigma(\overline{8_{20}}))| \leq
rk(\widehat{HF}(\Sigma(\overline{8_{20}})))
\]
because $b_{1}(\Sigma(\overline{8_{20}}))=0$ and
$|Spin^{c}(\Sigma(\overline{8_{20}}))|$ gives the Euler
characteristic of $\widehat{HF}(\Sigma(\overline{8_{20}}))$
according to Proposition 5.1 of \cite{OS7}. Furthermore,
\[
rk(\widehat{HF}(\Sigma(\overline{8_{20}}))) \leq
rk(\widetilde{Kh}(\overline{8_{20}}))
\]
\noindent (where both ranks refer to homology with $\mathbb{Z}_{2}$
coefficients) as there exists a spectral sequence with $E^{2}$ term
$\widetilde{Kh}(8_{20})$ with $\mathbb{Z}_{2}$ coefficients and
$E^{\infty}$ term
$\widehat{HF}(\Sigma(\overline{8_{20}});\mathbb{Z}_{2})$
(\cite{OS2}) and $rk(\widetilde{Kh}(\overline{8_{20}})) =
rk(\widetilde{Kh}(8_{20}))$. Lastly, according to Corollary 2 of
\cite{Kho2} the dimensions of the reduced homology groups for an
H-thin knot are given by the absolute values of the coefficients of
it's Jones polynomial and
\[
J(8_{20})=-q+2-q^{-1}+2q^{-2}-q^{-3}+q^{-4}-q^{-5}
\]
giving that
\[rk(\widetilde{Kh}(\overline{8_{20}})) =
rk(\widetilde{Kh}(8_{20})) = 9.
\]
\noindent Combining all the above relations in the order presented,
we get that
\[
rk(\widehat{HF}(\Sigma(\overline{8_{20}});\mathbb{Z}_{2}))=|Spin^{c}(\Sigma(\overline{8_{20}}))|=9
\]
which translates to the fact that $ \Sigma(\overline{8_{20}}) $ is
an L-space over $\mathbb{Z}_{2}$.

Similarly, $det(9_{46})=9=rk(\widetilde{Kh}(9_{46}))$,
$det(10_{137})=25=rk(\widetilde{Kh}(10_{137}))$,
$det(10_{140})=9=rk(\widetilde{Kh}(10_{140}))$ and the analogous
conclusions can be drawn for these examples.
\end{proof}

\begin{figure}[h]
\centering
\includegraphics{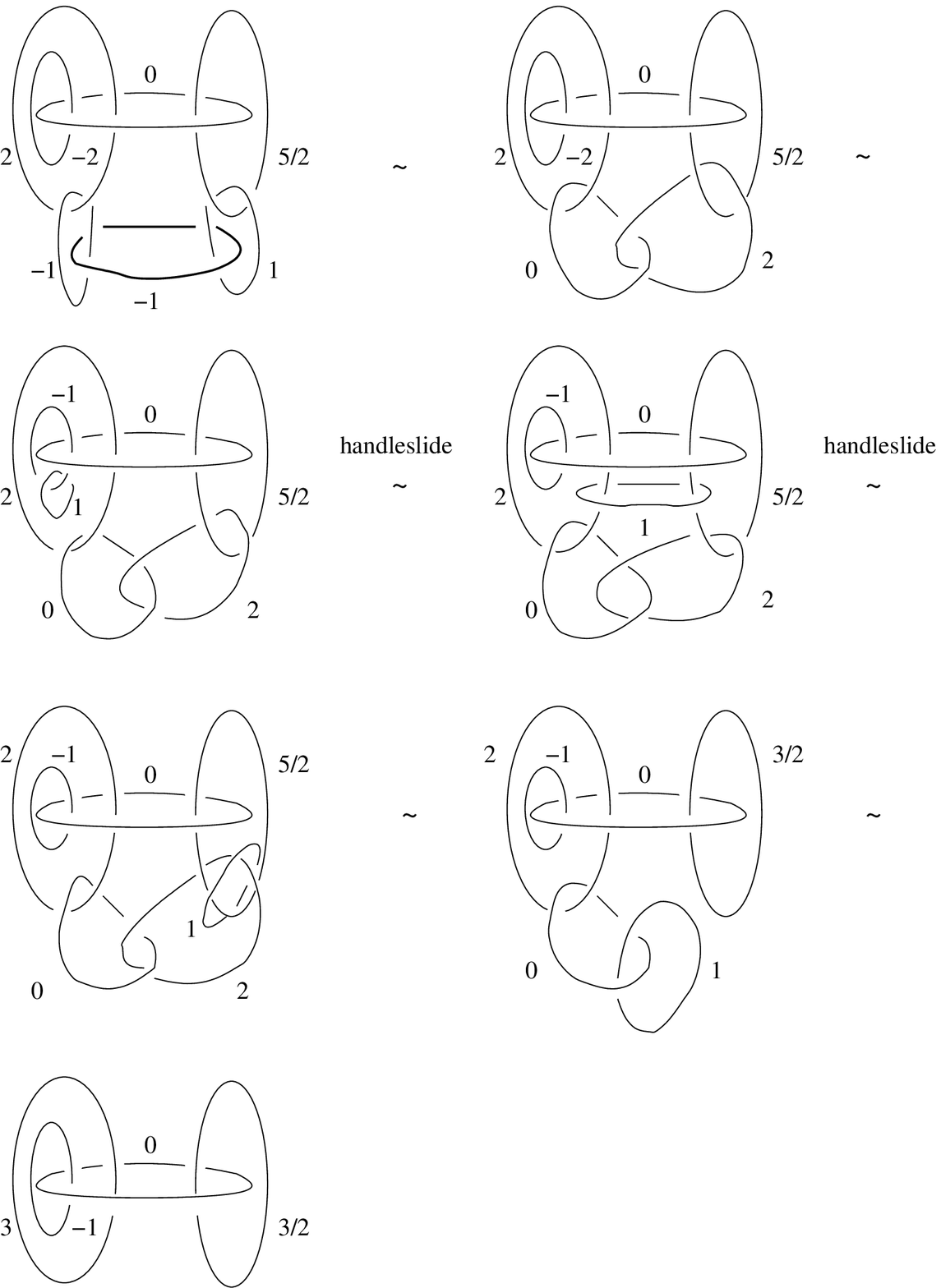}
\caption{2-handle addition for $Y_{1}$} \label{14}
\end{figure}


\begin{thebibliography}{99}
 \bibitem{BN} D. Bar-Natan, On Khovanov's categorification of the
 Jones polynomial, Algebraic and Geometric Topology 2-16 (2002) 337-370, arXiv:math.QA/0201043
 \bibitem{DBN} D. Bar-Natan, S. Morrison, The Knot Atlas:\\ http://katlas.math.toronto.edu/wiki/
 \bibitem{CH} A. Casson, J. Harer, Some homology lens spaces which
 bound rational homology balls, Pacific Journal of Mathematics, Vol.
 96, No. 1, 1981
 \bibitem{FS} R. Fintushel, R. Stern, Rational blowdowns of smooth
 4-manifolds, J. Differential Geometry, 46 (1997), 181-235
 \bibitem{GaS} D. Gay, A. Stipsicz, Symplectic rational blow-down
 along Seifert-fibered 3-manifolds, arXiv:math.SG/0703370 v1
 \bibitem{GS} R. Gompf, A. Stipcisz, 4-maifolds and Kirby Calculus,
 Graduate Studies in Mathematics, Volume 20, AMS
 \bibitem{Kho}M. Khovanov, A categorification of the Jones
 polynomial, Duke Math. J. 101 (2000), no. 3, 359--426, arXiv:math.QA/9908171 v2
 \bibitem{Kho2}M. Khovanov, Patterns in knot cohomology I, Experimental mathematics 12 (2003), no. 3,
 365--374, arXiv:math.QA/0201306 v1
 \bibitem{Li} R. Lickorish, An Introduction to Knot Theory, Springer
 \bibitem{In} C. Livingston, Table of Knot Invariants:\\ http://www.indiana.edu/$\thicksim$knotinfo/
 \bibitem{MO} C. Manolescu, P. Ozsv\'{a}th, On the Khovanov and knot
 Floer homologies of quasi-alternating links, arXiv:0708.3249v1
 \bibitem{OS5}P. Ozsv\'{a}th, Z. Szab\'{o}, Absolutely graded Floer homologies and
 intersections forms for four-manifolds with boundary, Advances in Mathematics 173 (2003) 179-261
 \bibitem{OS4}P. Ozsv\'{a}th, Z. Szab\'{o}, Holomorphic disks and topological invariants
 for closed 3-manifolds, Annals of Mathematics 159 (2004), no. 3,
 1027-1158.
 \bibitem{OS7}P. Ozsv\'{a}th, Z. Szab\'{o}, Holomorphic disks and
 three-manifolds invariants: Properties and applications, Annals of Mathematics 159 (2004) no. 3,
 1159-1245.
 \bibitem{OS6}P. Ozsv\'{a}th, Z. Szab\'{o}, Holomorphic triangles
 and invariants for smooth four-manifolds, arXiv:math.GT/0401426 v2
 \bibitem{OS3} P. Ozsv\'{a}th, Z. Szab\'{o}, On Knot Floer homology
 and lens space surgeries, Topology 44 (2005), no. 6, 1281-1300, arXiv:math.GT/0303017 v2
 \bibitem{OS1} P. Ozsv\'{a}th, Z. Szab\'{o}, On the Floer
 homology of plumbed three-manifolds, Geometry and Topology, Volume
 7 (2003) 185-224
 \bibitem{OS2} P. Ozsv\'{a}th, Z. Szab\'{o}, On the Heegaard Floer
 homology of branched double-covers, Adv. in Mathematics Volume 194 (2005), pages 1-33, arXiv:math.GT/0309170 v1
 \bibitem{P2} J. Park, Seiberg-Witten invariants of generalised
 rational blow-downs, Bull. Austral. Math. Soc., Vol. 56 (1997),
 363-384
 \bibitem{P} J. Park, Simply connected symplectic 4-manifolds with
 $b_{2}^{+}=1$ and $c_{1}^{2}=2$, arXiv:math.GT/0311395 v3
 \bibitem{PSS} J. Park, A. Stipsicz, Z. Szab\'{o}, Exotic smooth structures on
 $\mathbb{C}P^{2}\sharp 5 \overline{\mathbb{C}P}^{2}$, arXiv:math.GT/0412216 v3
 \bibitem{Ro} L. Roberts, Rational blow downs in Heegaard-Floer
 homology, arXiv:math.GT/0607675 v1
 \bibitem{SS} A. Stipsicz, Z. Szab\'{o}, An exotic smooth structure on $\mathbb{C}P^{2}\sharp 6
 \overline{\mathbb{C}P}^{2}$, Geometry and Topology, Volume 9 (2005)
 813-832
 \bibitem{SSW} A. Stipsicz, Z. Szab\'{o}, J. Wahl,
 Rational blow-downs and smoothings of surface singularities,
 arXiv:math.GT/0611157 v2
 \end{thebibliography}
\end{document}